\documentclass{amsart}

\usepackage{amssymb, amsthm, fancyhdr,amsmath,enumerate,graphicx,multicol}
\usepackage{subcaption}
\usepackage[inline]{enumitem}
\usepackage[numbers]{natbib}
\usepackage{color}
\usepackage{mathrsfs}
\usepackage{bbm}
\usepackage{appendix}
\usepackage[colorlinks=true,linkcolor=blue]{hyperref}
\usepackage{tcolorbox}
\usepackage{mathtools}
\mathtoolsset{showonlyrefs=true}
\usepackage[left=1.2in,top=1.3in,right=1.2in,bottom=1.3in]{geometry}
\usepackage{tikz-cd}
\usepackage{accents}

\DeclareMathOperator{\atv}{ATV}
\DeclareMathOperator{\tv}{TV}
\DeclareMathOperator{\cpl}{Cpl}
\DeclareMathOperator{\cplbc}{Cpl_{bc}}

\DeclareMathOperator{\id}{id}

\def\a{\mathcal{A}}

\def\sp{\mathcal{P}}

\def\x{\mathcal{X}}

\def\w{\mathcal{W}}

\def\P{\mathbb{P}}

\def\R{\mathbb{R}}

\def\ep{\varepsilon}

\newcommand{\abs}[1]{\left\vert#1\right\vert}

\newcommand{\ind}[1]{\mathbbm{1}_{#1}}
\newcommand{\indep}{\mathrel{\perp\!\!\!\perp}}
\newcommand{\ent}[2]{H\left(#1\,|\,#2\right)}

\newlist{anumerate}{enumerate}{1}
\setlist[anumerate,1]{label=(\alph*)}

\newtheorem{thm}{Theorem}
\newtheorem{cor}[thm]{Corollary}
\newtheorem{df}[thm]{Definition}

\newtheorem{lem}[thm]{Lemma}

\theoremstyle{definition}

\newtheorem{rmk}[thm]{Remark}

\begin{document}

\title[Adapted $T_1$ inequality]{On a $T_1$ Transport inequality for the adapted Wasserstein distance}
\author[Jonghwa Park]{Jonghwa Park}\thanks{The author would like to thank Martin Larsson and Johannes Wiesel for their valuable comments and insightful suggestions.}
\address{Jonghwa Park\newline
Carnegie Mellon University, Department of Mathematical Sciences\newline
jonghwap@andrew.cmu.edu}
\date{}


\begin{abstract}
The $L^1$ transport-entropy inequality (or $T_1$ inequality), which bounds the $1$-Wasserstein distance in terms of the relative entropy, is known to characterize Gaussian concentration. To extend the $T_1$ inequality to laws of discrete-time processes while preserving their temporal structure, we investigate the \emph{adapted} $T_1$ inequality which relates the $1$-adapted Wasserstein distance to the relative entropy. Building on the Bolley--Villani inequality, we establish the adapted $T_1$ inequality under the same moment assumption as the classical $T_1$ inequality.  
\end{abstract}

\maketitle

\section{Introduction}

Throughout this work, we consider Borel probability measures on a Polish space $\x$ equipped with a metric $d$. We denote by $\sp(\x)$ the set of all Borel probability measures on $\x$.

Over the last decades, the $L^p$ \textit{transport-entropy} inequality, which controls the Wasserstein distances by the relative entropy, has been extensively studied due to its deep connections with concentration of measure and deviation inequalities; see \cite{gozlan2010transport} for a general overview. In this article, we extend the transport-entropy inequality to laws of discrete-time processes by using the \textit{adapted Wasserstein distance}, a variant of the Wasserstein distance, particularly well suited for measuring distances between laws of processes. More precisely, we aim to establish a transport inequality that relates the adapted Wasserstein distance to the relative entropy under minimal assumptions on the reference measure.

Before we proceed, let us recall some basic definitions. For $1\le p<\infty$, the $p$-\textit{Wasserstein distance} between $\mu\in \sp(\x)$ and $\nu\in \sp(\x)$ is defined via
\begin{align}
    \w_p(\mu, \nu)
    :=\inf_{\pi \in \cpl(\mu, \nu)}\left(\int_{\x\times \x} d(x,y)^p \pi(dx, dy)\right)^{1/p}\label{eq:wasserstein distance}
\end{align}
where $\cpl(\mu, \nu)$ is the set of all couplings between $\mu$ and $\nu$. Note that $\w_p$ is a metric on $\sp_p(\x)$, the set of all $\rho\in \sp(\x)$ such that $\int_{\x}d(x, x_0)^p \rho(dx)<\infty$. We say that $\mu\in \sp_p(\x)$ satisfies the $L^p$ \textit{transport-entropy} inequality (also called the $T_p$ inequality) if there exists some constant $C>0$ such that
\begin{align}
    \w_p(\mu, \nu)\le C\sqrt{2\ent{\nu}{\mu}} \text{ for all }\nu\in \sp(\x).\label{eq:w_pH inequality}
\end{align}
If \eqref{eq:w_pH inequality} holds with constant $C$, 
then $\mu$ is said to satisfy the $T_p(C)$ inequality. Here, $\ent{\cdot}{\mu}$ denotes the relative entropy with respect to $\mu$, defined as
\begin{align}
    \ent{\nu}{\mu}
    :=\begin{cases}
        \int_{\x} \log\left(\frac{d\nu}{d\mu}\right)d\nu & \text{ if } \nu\ll \mu,\\
        +\infty & \text{ otherwise.}
    \end{cases}
\end{align}

There are well-established connections between $T_p$ inequalities and concentration of measure. In particular, when $p=1$, it is shown in \cite{bobkov1999exponential, djellout2004} that the $T_1$ inequality is equivalent to Gaussian concentration.

\begin{thm}[Theorem $22.10$ in \cite{villani2008optimal}]\label{thm:gaussian concentration} Let $\mu\in \sp_1(\x)$ and $C>0$. Each of the following statements implies the next one:
\begin{enumerate}[label=(\alph*)]
    \item $\mu$ satisfies the $T_1(C)$ inequality.
    \item $\mu$ exhibits Gaussian concentration: for any $\ep>0$ and $1$-Lipschitz $f:\x\to \R$ such that $f\in L^1(\mu)$,
    \begin{align}
        \mu(\{x\in \x : f(x)\ge \int_{\x} f d\mu+\ep\})\le e^{-\frac{\ep^2}{2C^2}}.
    \end{align}
    \item $\mu$ has a finite quadratic exponential moment: for any $x_0\in \x$ and $\alpha\in (0, \frac{1}{2C^2})$,
    \begin{align}
        \int_{\x}e^{\alpha d(x, x_0)^2}\mu(dx)<\infty.
    \end{align}
    \item For any $x_0\in \x$ and $\alpha\in (0, \frac{1}{2C^2})$, $\mu$ satisfies the $T_1(C')$ inequality where
    \begin{align}
        C'=\frac{1}{\sqrt{\alpha}}\left(1+\log \int_{\x} e^{\alpha d(x, x_0)^2}\mu(dx)\right)^{1/2}<\infty.
    \end{align}
\end{enumerate}
\end{thm}

The last implication from $(c)$ to $(d)$ is a special case of the following bound on the \textit{weighted} total variation, obtained by choosing $\varphi(x)=\sqrt{\alpha}d(x, x_0)$ below.

\begin{thm}[Theorem 2.1 in Bolley--Villani \cite{bolley2005weighted}]\label{thm:bolley villani}
Let $\mu, \nu\in \sp(\x)$ and $\varphi:\x\to [0,\infty)$ be a Borel measurable function. Then
\begin{align}
    \tv_{\varphi}(\mu, \nu)
    :=\int_{\x} \varphi(x)\abs{\mu-\nu}(dx)
    \le \left(1+\log \int_{\x}e^{\varphi(x)^2}\mu(dx)\right)^{1/2}\sqrt{2 \ent{\nu}{\mu}}.
\end{align}
Here, $\abs{\mu-\nu}$ denotes the total variation measure of $\mu-\nu$ defined via
\begin{align}
    \abs{\mu-\nu}(A)=\int_{A}\abs{\frac{d\mu}{d(\mu+\nu)}-\frac{d\nu}{d(\mu+\nu)}} d(\mu+\nu)
\end{align}
for all measurable $A$.
\end{thm}

Our goal is to extend the Bolley--Villani bound to the adapted Wasserstein distance when $\mu$ and $\nu$ are laws of discrete-time processes. This will follow from the \textit{adapted} version of Theorem~\ref{thm:bolley villani} which we present in Theorem~\ref{thm:adapted Bolley Villani}.

\subsection{Main results}
For the remainder of this article, we consider laws of processes. To that end, we set $\x=\prod_{t=1}^T \x_t$ where $\x_t$ is a Polish space and interpret $\mu\in \sp(\x)$ as a law of a process $X=(X_t)_{t=1}^T$ where $T$ is the number of time steps and each $X_t$ is a $\x_t$-valued random variable. The metric $d$ on $\x$ is any metric that metrizes the product topology.

Recently, it has been shown that the Wasserstein distance is not an adequate metric for time-dependent optimization problems that take the flow of information into account. For example, the value functions of optimal stopping problems, superhedging problems and stochastic programming problems are not continuous with respect to the Wasserstein distance. See \cite{backhoff2020all, pflug2012distance, backhoff2020adapted, bartl2023sensitivity, backhoff2017causal, acciaio2024designing} for further examples and detailed discussions.

The adapted Wasserstein distance has been developed to address this issue. In contrast to the classical Wasserstein distance \eqref{eq:wasserstein distance} which minimizes the cost over all possible couplings $\pi$, the adapted Wasserstein distance is defined as the minimal cost over \textit{bicausal} couplings which we describe below.

\begin{df}[Bicausal couplings] Let $\mu, \nu\in \sp(\x)$. A coupling $\pi \in \cpl(\mu, \nu)$ is bicausal if for $(X,Y)\sim \pi$ and $t\in \{1,2,\ldots, T-1\}$,
\begin{align}
    Y_1, \ldots, Y_t \indep_{X_{1}, \ldots, X_{t}} X \text{ and }
    X_1, \ldots, X_t \indep_{Y_{1}, \ldots, Y_{t}} Y \text{ under } \pi.
\end{align}
We denote by $\cplbc(\mu, \nu)$ the set of all bicausal couplings between $\mu$ and $\nu$.
\end{df}

\begin{df}[The adapted Wasserstein distance]
Let $\mu, \nu\in \sp(\x)$ and $1\le p<\infty$. The $p$-adapted Wasserstein distance between $\mu$ and $\nu$ is defined via
\begin{align}
    \a\w_p(\mu, \nu)
    =\inf_{\pi \in \cplbc(\mu, \nu)}\left(\int_{\x\times \x} d(x,y)^p \pi(dx, dy)\right)^{1/p}.
\end{align}
\end{df}
The adapted Wasserstein distance $\a\w_p$ is well suited for studying optimization problems in which the flow of information plays a central role (see \cite{backhoff2020all}). For example, it induces the smallest topology under which optimal stopping problems have continuous value functions. While the topology induced by $\a\w_p$ is well understood, the metric
$\a\w_p$ remains difficult to control quantitatively. This poses a significant limitation in applications where quantitative control of $\a\w_p$ is required.

In this paper, we address this limitation by comparing the adapted weighted total variation with the relative entropy. More precisely, Theorem~\ref{thm:adapted Bolley Villani} extends the Bolley--Villani bound (Theorem~\ref{thm:bolley villani}) and shows that the adapted weighted total variation can be bounded in terms of the relative entropy, under suitable moment assumptions. This result provides a new quantitative tool linking an adapted optimal cost to the relative entropy, a fundamental concept in information theory. In particular, Corollary~\ref{cor:aw_1 H inequality} shows that a measure $\mu$ satisfies the adapted $T_1$ inequality if and only if $\mu$ has subgaussian tails.

We now introduce the adapted weighted total variation and present our main
result, Theorem~\ref{thm:adapted Bolley Villani}.

\begin{df}\label{def:adapted tv}
Let $\mu, \nu\in \sp(\x)$ and $\varphi:\x\to [0,\infty)$ be a Borel measurable function. We define the adapted weighted total variation as
\begin{align}
    \atv_{\varphi}(\mu, \nu)
    =\inf_{\pi \in \cplbc(\mu, \nu)}\int_{\x\times \x} (\varphi(x)+\varphi(y))\ind{\{x\neq y\}}\pi(dx, dy).
\end{align}
\end{df}
\begin{rmk}
Note that the weighted total variation can be expressed as
\begin{align}
    \tv_{\varphi}(\mu, \nu)
    =\inf_{\pi \in \cpl(\mu, \nu)}\int_{\x\times \x} (\varphi(x)+\varphi(y))\ind{\{x\neq y\}} \pi(dx, dy)
\end{align}
where the infimum is attained when $\pi(dx, dy)=(\id, \id)_{\#}(\mu\wedge \nu)$ on $\{x=y\}$. For example, see \cite[Section 2]{Torgny1999} or \cite[Lemma 3.1]{acciaio2025estimating}. Here, $\mu\wedge \nu$ is the minimum between $\mu$ and $\nu$, defined by
\begin{align}
    \mu\wedge \nu(A)=\int_{A} \left(\frac{d\mu}{d(\mu+\nu)}\wedge \frac{d\nu}{d(\mu+\nu)}\right) d(\mu+\nu)
\end{align}
for all measurable $A$. In this context, $\atv_{\varphi}$ can be viewed as a natural extension of $\tv_{\varphi}$. The study of $\atv_{\varphi}$ is not new and related works will be discussed in the next section.    
\end{rmk}

\begin{thm}\label{thm:adapted Bolley Villani}
Let $\mu, \nu\in \sp(\x)$ and $\varphi:\x\to [0,\infty)$ be a Borel measurable function. Then
\begin{align}
    \atv_{\varphi}(\mu, \nu)
    \le (2\sqrt{T}+1)
    \left(1+\log \int_{\x}e^{\varphi(x)^2}\mu(dx)\right)^{1/2}
    \sqrt{2 \ent{\nu}{\mu}}.\label{eq:adapted Bolley Villani}
\end{align}
\end{thm}

The Bolley--Villani bound in Theorem~\ref{thm:bolley villani} is a \textit{one-step} bound that applies at each time step. In the adapted setting, we inductively apply the Bolley--Villani bound to the conditional distributions of $\mu$ to obtain Theorem~\ref{thm:adapted Bolley Villani}. Thanks to the chain rule of entropy, the Bolley--Villani bound propagates along time without any loss of information. The proofs are provided in Section~\ref{sec:proofs}.

By properly choosing the weight function $\varphi$, we obtain the transport-entropy inequality for $\a\w_1$. In view of Theorem~\ref{thm:gaussian concentration}, the adapted $T_1$ inequality is equivalent to Gaussian concentration. 
\begin{cor}[The adapted $T_1$ inequality]\label{cor:aw_1 H inequality}
Let $\mu\in \sp_1(\x)$ and $C>0$. Each of the following statements implies the next one:
\begin{enumerate}[label=(\alph*)]
    \item $\mu$ satisfies the adapted $T_1(C)$ inequality, i.e.,
    \begin{align}
        \a\w_1(\mu, \nu)\le C\sqrt{2 \ent{\nu}{\mu}} \text{ for all } \nu\in \sp(\x).
    \end{align}
    \item $\mu$ satisfies the $T_1(C)$ inequality.
    \item For any $x_0\in \x$ and $\alpha\in (0, \frac{1}{2C^2})$, $\mu$ satisfies the adapted $T_1(C')$ inequality where
    \begin{align}
        C'=\frac{2\sqrt{T}+1}{\sqrt{\alpha}}\left(1+\log \int_{\x} e^{\alpha d(x, x_0)^2} \mu(dx)\right)^{1/2}.
    \end{align}
\end{enumerate}
\end{cor}

By the definition of the adapted Wasserstein distance $\a\w_1$, it is clear that $\w_1\le \a\w_1$ and thus the implication from $(a)$ to $(b)$ follows immediately. The main content of the above corollary lies in the last direction: $(b)$ implies $(c)$.

Similarly, we obtain the following estimates on $\a\w_p$ for $p>1$ in terms of the relative entropy which is of independent interest.
\begin{cor}[The general case]\label{cor:general p}
Let $p>1$ and $\mu\in \sp(\x)$ such that $\int_{\x} e^{\alpha d(x, x_0)^{2p}}\mu(dx)<\infty$ for some $x_0\in \x$ and $\alpha>0$. Then
\begin{align}
    \a\w_p(\mu, \nu)^p
    \le \frac{2^{p-1}(2\sqrt{T}+1)}{\sqrt{\alpha}}\left(1+\log \int_{\x} e^{\alpha d(x, x_0)^{2p}} \mu(dx)\right)^{1/2}
    \sqrt{2 \ent{\nu}{\mu}} \text{ for all } \nu\in \sp(\x).
\end{align}
\end{cor}

\subsection{Related works}
In this section, we present existing results and compare them with our results.

Let us recall that if the reference measure $\mu$ has independent marginals, i.e., $\mu=\otimes_{t=1}^T \gamma_t$ and each $\gamma_t$ satisfies the $T_1$ inequality \eqref{eq:w_pH inequality} with a constant $C>0$, then the tensorization argument yields
\begin{align}
    \w_1(\mu, \nu)\le \sqrt{T}C\sqrt{2\ent{\nu}{\mu}} \text{ for all } \nu\in \sp(\x)
\end{align}
where a metric $d$ on $\x=\prod_{t=1}^T \x_t$ is the $\ell^1$ sum of the componentwise metrics $d_{\x_t}$ on $\x_t$, i.e., $d(x,y)=\sum_{t=1}^Td_{\x_t}(x_t, y_t)$. As pointed out in \cite[Proposition 5.9]{backhoff2017causal}, the tensorization argument actually implies
\begin{align}
    \a\w_1(\mu, \nu)\le \sqrt{T}C\sqrt{2\ent{\nu}{\mu}} \text{ for all } \nu\in \sp(\x).\label{eq:independent}
\end{align}

There have been several attempts \cite{djellout2004, bolley2005weighted, villani2008optimal} to extend this result to the dependent case. Since standard tensorization does not apply, most of these results assume certain regularity conditions on the kernels of $\mu$. More recently, this line of work has been revisited and further developed by \cite{backhoff2017causal}. To illustrate, let $X=(X_t)_{t=1}^T$ be a $\R^T$-valued process whose law is $\mu\in \sp(\R^T)$ and denote the (regular) conditional distribution of $X_{t+1}$ given $(X_1, \ldots, X_t)=(x_1, \ldots, x_t)$ by
\begin{align}
    \mu^{x_1, \ldots, x_{t}}(\cdot)
    :=\P(X_{t+1}\in \cdot \,|\, X_1=x_1, \ldots, X_t=x_t)\in \sp(\R).
\end{align}
In \cite{backhoff2017causal}, it is assumed that the reference measure $\mu$ has Lipschitz kernels, i.e.,
\begin{align}
    \w_1(\mu^{x_1, \ldots, x_t}, \mu^{y_1, \ldots, y_t})
    \le C_{\text{lip}}\sum_{s=1}^{t}\abs{x_s-y_s}, \quad t\in \{1,2, \ldots, T-1\},
\end{align}
and satisfies the moment condition $\int e^{a_1 \abs{x_1}^2}\mu(dx_1)<\infty$, along with
\begin{align}
    \mu\text{-ess sup }\int e^{a_{t+1} \abs{x_{t+1}}^2}\mu^{x_1, \ldots, x_{t}}(dx_{t+1})<\infty, \quad t\in \{1,2, \ldots, T-1\}.\label{eq:lip mom}
\end{align}
Under these assumptions, it is established that
\begin{align}
    \a\w_1(\mu, \nu)\le (C_{\text{lip}})^T C_{\text{mom}}\sqrt{2\ent{\nu}{\mu}}
    \text{ for all } \nu\in \sp(\R^T)
\end{align}
where $C_{\text{mom}}$ is some positive constant computed from the moment assumption \eqref{eq:lip mom}. Compared to our result (Corollary~\ref{cor:aw_1 H inequality}), it relies on strictly stronger assumptions on $\mu$. Moreover, the bounding constant depends exponentially on $T$ which is not consistent with the independent case \eqref{eq:independent} where tensorization yields sublinear growth in $T$.

Our result is based on the Bolley--Villani bound (Theorem~\ref{thm:bolley villani}) and does not extend the tensorization technique beyond the independent case. We aim to investigate its generalization in future research.

As a special case of the $T_1$ inequality, we recall the classical Pinsker inequality:
\begin{align}
    \tv(\mu, \nu):=2\inf_{\pi\in\cpl(\mu, \nu)}\pi(x\neq y)
    \le \sqrt{2 \ent{\nu}{\mu}}.
\end{align}
Recently, the \textit{adapted} Pinsker inequality:
\begin{align}
    \atv(\mu, \nu)
    :=2\inf_{\pi \in \cplbc(\mu, \nu)}\pi(x\neq y)
    \le \sqrt{T}\sqrt{2 \ent{\nu}{\mu}}\label{eq:adapted pinsker}
\end{align}
was established in \cite{beiglbock2025pinsker}. As pointed out in \cite{beiglbock2025pinsker}, this is sharp. However, applying the constant function $\varphi\equiv \sqrt{\alpha}$ to our result (Theorem~\ref{thm:adapted Bolley Villani}) yields
\begin{align}
    \atv(\mu, \nu)
    =\lim_{\alpha\to \infty}\frac{1}{\sqrt{\alpha}}\atv_{\sqrt{\alpha}}(\mu, \nu)
    &\le \lim_{\alpha\to \infty}(2\sqrt{T}+1)\left(\frac{1}{\alpha}+1\right)^{1/2}\sqrt{2\ent{\nu}{\mu}}\\
    &=(2\sqrt{T}+1)\sqrt{2\ent{\nu}{\mu}}
\end{align}
with the non-optimal constant $2\sqrt{T}+1$. This lack of sharpness is not surprising as the Bolley--Villani bound (Theorem~\ref{thm:bolley villani}), which will play a central role in the proof of Theorem~\ref{thm:adapted Bolley Villani}, is not sharp in general. While the constant is not optimal, our result preserves the correct sublinear dependence on $T$.

An alternative approach to studying transport-entropy inequalities for the adapted Wasserstein distance is to bound adapted distances by their non-adapted counterparts for which efficient transport inequalities are known. One of the earliest works in this direction is \cite{eckstein2024computational} which shows that
the adapted total variation can be bounded above by the total variation. Precisely, Eckstein--Pammer \cite{eckstein2024computational} established that
\begin{align}
    \atv(\mu, \nu)
    \le C\tv(\mu, \nu)\label{eq:atv and tv}
\end{align}
holds with $C=2^T-1$. The optimal constant $C=2T-1$ was later obtained by Acciaio et al. \cite{acciaio2025estimating}. In combination with the classical Pinsker inequality, this implies
\begin{align}
    \atv(\mu, \nu)
    \le (2T-1)\tv(\mu,\nu)
    \le (2T-1)\sqrt{2\ent{\nu}{\mu}}.
\end{align}
Note that this approach does not yield the optimal dependence on $T$. 

Acciaio et al. \cite{acciaio2025estimating} further showed that if $\mu\in \sp((\R^d)^T)$ satisfies
\begin{align}
    \mu\text{-ess sup} \frac{1}{1+\sum_{s=1}^{t}\abs{x_s}}\int \abs{x_{t+1}}\mu^{x_1, \ldots, x_{t}}(dx_{t+1})<\infty \text{ for all } t\in \{1,2,\ldots, T-1\},\label{eq:watv assumption}
\end{align}
then
\begin{align}
    \atv_{\abs{\cdot}}(\mu, \nu)
    \le C^T\tv_{\abs{\cdot}}(\mu, \nu)\label{eq:watv wtv}
\end{align}
for some constant $C>0$ that arises from the assumption \eqref{eq:watv assumption}. This estimate, when combined with the Bolley--Villani bound (Theorem~\ref{thm:bolley villani}), naturally leads to the following transport inequality:
\begin{align}
    \a\w_1(\mu, \nu)\le \frac{C^T}{\sqrt{\alpha}}\left(1+\log \int e^{\alpha\abs{x}^2}\mu(dx)\right)^{1/2} \sqrt{2 \ent{\nu}{\mu}} \text{ for all } \nu\in \sp((\R^d)^T).
\end{align}
While the estimate \eqref{eq:watv wtv} is of independent interest, our result improves upon the transport inequality deduced from it. In particular, the exponential dependence on $T$ is reduced to $\sqrt{T}$, up to a logarithmic factor. Moreover, we establish the adapted $T_1$ inequality under the minimal assumption without requiring any regularity conditions such as \eqref{eq:watv assumption}.

\section{Proofs}\label{sec:proofs}
Let us first set up some notation. Recall that $\x=\prod_{t=1}^T \x_t$. We use the shorthand notation $\x_{s:t}=\x_s\times\cdots\times \x_t$ for $1\le s\le t\le T$. Similarly, for $x=(x_1, \ldots, x_T)\in \x$, we denote $x_{s:t}:=(x_s, \ldots, x_t)$. For $\mu\in \sp(\x)$, $\mu_t\in \sp(\x_t)$ is the projection of $\mu$ onto the $t$-th coordinate and $\mu_{1:t}\in \sp(\x_1\times\cdots\times \x_t)$ is the projection of $\mu$ onto the first $t$ coordinates, i.e.,
\begin{align}
    \mu_t(A):=\mu(\x_1\times\cdots\times \x_{t-1}\times A\times \x_{t+1}\times \cdots \times \x_{T}),\quad
    \mu_{1:t}(A):=\mu(A\times \x_{t+1}\times \cdots \times \x_{T}).
\end{align}
For $x\in \x$ and $\mu\in \sp(\x)$, we define $\mu^{x_{1:t}}\in \sp(\x_{t+1})$ for $t\in \{1,\ldots, T-1\}$ via
\begin{align}
    \mu^{x_{1:t}}(dx_{t+1})
    =\P(X_{t+1}\in dx_{t+1} \,|\, X_{1}=x_1, \ldots, X_t=x_t).
\end{align}
where $X=(X_t)_{t=1}^T\sim \mu$. Similarly, we define $\overline{\mu}^{x_{1:t}}\in \sp(\x_{t+1:T})$ for $t\in \{1,\ldots, T-1\}$ via
\begin{align}
    \overline{\mu}^{x_{1:t}}(dx_{t+1:T})
    =\P(X_{t+1:T}\in dx_{t+1:T} \,|\, X_{1}=x_1, \ldots, X_t=x_t)
\end{align}
for $X\sim \mu$ or equivalently,
\begin{align}
    \overline{\mu}^{x_{1:t}}(dx_{t+1:T})
    =\mu^{x_{1:T-1}}(dx_{T})\cdots \mu^{x_{1:t}}(dx_{t+1}).
\end{align}
Note that we have the following disintegrations
\begin{align}
    \mu(dx)
    =\mu^{x_{1:T-1}}(dx_T)\cdots\mu^{x_1}(dx_2)\mu_1(dx_1)
    =\overline{\mu}^{x_{1:t}}(dx_{t+1:T})\mu_{1:t}(dx_{1:t}).
\end{align}

We denote the identity map on $\x_t$ by $\id_{\x_t}:\x_t\to \x_t$. For any positive measure $\rho$, the pushforward measure of $\rho$ through a measurable map $f$ is denoted by $f_{\#}\rho$.

If $\gamma$ is a finite signed measure, we write its Jordan decomposition as $\gamma=\gamma_{+}-\gamma_{-}$, where $\gamma_{+}$ and $\gamma_{-}$ are the positive and negative variations of $\gamma$, respectively.

\begin{lem}\label{lem:tv+sum}
Let $\mu, \nu\in \sp(\x)$ and $\varphi:\x\to [0,\infty)$ be a Borel measurable function such that
\begin{align}
    \int_{\x}\varphi(x)(\mu+\nu)(dx)<\infty.
\end{align}
For each $j\in\{1,2,\ldots, T\}$, define a positive measure $\gamma^{(j)}$ on $\x$ as follows:
\begin{align}
    \gamma^{(j)}(dx)
    :=\overline{\mu}^{x_{1:j}}(dx_{j+1:T})\abs{\mu^{x_{1:j-1}}-\nu^{x_{1:j-1}}}(dx_j)\mu_{1:j-1}\wedge\nu_{1:j-1}(dx_{1:j-1})
\end{align}
for $2\le j\le T-1$ and
\begin{align}
    &\gamma^{(1)}(dx)
    :=\overline{\mu}^{x_1}(dx_{2:T})\abs{\mu_1-\nu_1}(dx_1),\\
    &\gamma^{(T)}(dx)
    :=\abs{\mu^{x_{1:T-1}}-\nu^{x_{1:T-1}}}(dx_T)\mu_{1:T-1}\wedge\nu_{1:T-1}(dx_{1:T-1}).
\end{align}
Then we have
\begin{align}
    \atv_{\varphi}(\mu, \nu)
    \le \tv_{\varphi}(\mu, \nu)
    +2\sum_{j=1}^{T} \int_{\x}\varphi(x)\gamma^{(j)}(dx).\label{eq:tv+sum result}
\end{align}
\end{lem}

\begin{proof}
For $\gamma\in \cplbc(\mu, \nu)$ that will be chosen later, note that
\begin{align}
    \int_{\x\times \x} \varphi(y)\ind{\{x\neq y\}}\gamma(dx, dy)
    &=\int_{\x}\varphi(y)\nu(dy)
    -\int_{\x\times \x}\varphi(y)\ind{\{x=y\}}\gamma(dx, dy)\\
    &=\int_{\x}\varphi(y)\nu(dy)
    -\int_{\x\times \x}\varphi(x)\ind{\{x=y\}}\gamma(dx, dy)\\
    &=\left(\int_{\x}\varphi(y)\nu(dy)-\int_{\x}\varphi(x)\mu(dx)\right)
    +\int_{\x\times \x} \varphi(x)\ind{\{x\neq y\}}\gamma(dx, dy).
\end{align}
Using
\begin{align}
    \int_{\x}\varphi(y)\nu(dy)-\int_{\x}\varphi(x)\mu(dx)
    \le \int_{\x}\varphi(x)\abs{\nu-\mu}(dx)
    =\tv_{\varphi}(\mu, \nu),
\end{align}
we have
\begin{align}
    \int_{\x\times \x} (\varphi(x)+\varphi(y))\ind{\{x\neq y\}}\gamma(dx, dy)
    \le \tv_{\varphi}(\mu, \nu)
    +2\int_{\x\times \x} \varphi(x)\ind{\{x\neq y\}}\gamma(dx, dy).
\end{align}
Let us consider a partition $(E_j)_{j=1}^T$ of the set $\{x\neq y\}$ which is defined as $E_1:=\{x_1\neq y_1\}$ and
\begin{align}
    E_j:=\{x_j\neq y_j, x_{1:j-1}=y_{1:j-1}\} \text{ for } 2\le j\le T.
\end{align}
In particular,
\begin{align}
    \int_{\x\times \x} (\varphi(x)+\varphi(y))\ind{\{x\neq y\}}\gamma(dx, dy)
    \le \tv_{\varphi}(\mu, \nu)
    +2\sum_{j=1}^T\int_{E_j} \varphi(x)\gamma(dx, dy).\label{eq:tv+sum bound}
\end{align}

Now, we construct $\gamma$. First, define $\gamma_{1,1}\in \cpl(\mu_1, \nu_1)$ via
\begin{align}
    \gamma_{1,1}
    :=(\id_{\x_1}, \id_{\x_1})_{\#}\eta_1
    +\frac{(\mu_1-\nu_1)_{+}\otimes (\nu_1-\mu_1)_{+}}{(\mu_1-\nu_1)_{+}(\x_1)}
\end{align}
where $\eta_1:=\mu_1\wedge \nu_1$. Recall that $(\mu_1-\nu_1)_{+}$ and $(\nu_1-\mu_1)_{+}$ denote the positive variations of $\mu_1-\nu_1$ and $\nu_1-\mu_1$, respectively. For each $t\in \{1,2,\ldots, T-1\}$ and $x_{1:t}, y_{1:t}\in \x_{1:t}$, we define $\gamma^{x_{1:t}, y_{1:t}}\in \cpl(\mu^{x_{1:t}}, \nu^{y_{1:t}})$ via
\begin{align}
    \gamma^{x_{1:t}, y_{1:t}}
    :=\begin{cases}
        (\id_{\x_{t+1}}, \id_{\x_{t+1}})_{\#}\eta^{x_{1:t}}
        +\frac{(\mu^{x_{1:t}}-\nu^{x_{1:t}})_{+}\otimes (\nu^{x_{1:t}}-\mu^{x_{1:t}})_{+}}{(\mu^{x_{1:t}}-\nu^{x_{1:t}})_{+}(\x_{t+1})} & \text{ if } x_{1:t}=y_{1:t},\\[1.5ex]
        \mu^{x_{1:t}}\otimes \nu^{y_{1:t}} & \text{ if } x_{1:t}\neq y_{1:t}
    \end{cases}
\end{align}
where $\eta^{x_{1:t}}:=\mu^{x_{1:t}}\wedge \nu^{x_{1:t}}$. We define $\gamma\in \sp(\x\times \x)$ from the following formula:
\begin{align}
    \gamma(dx, dy)
    :=\gamma^{x_{1:T-1}, y_{1:T-1}}(dx_T, dy_T)\cdots
    \gamma^{x_{1}, y_{1}}(dx_2, dy_2)
    \gamma_{1,1}(dx_1, dy_1).
\end{align}
Note that $\gamma\in \cplbc(\mu, \nu)$ since each kernel of $\gamma$ couples the corresponding kernels of $\mu$ and $\nu$; see \cite[Proposition 5.1]{backhoff2017causal}.

From \eqref{eq:tv+sum bound}, it suffices to show that for all $j\in \{1,2,\ldots, T\}$,
\begin{align}
    \int_{E_j}\varphi(x)\gamma(dx, dy)
    \le \int_{\x}\varphi(x)\gamma^{(j)}(dx).
\end{align}
By the construction of $\gamma$, we have that on $E_1$,
\begin{align}
    \gamma(dx, dy)
    =\gamma^{x_{1:T-1}, y_{1:T-1}}(dx_T, dy_T)\cdots
    \gamma^{x_{1}, y_{1}}(dx_2, dy_2)
    \frac{(\mu_1-\nu_1)_{+}(dx_1)(\nu_1-\mu_1)_{+}(dy_1)}{(\mu_1-\nu_1)_{+}(\x_1)}.
\end{align}
Hence
\begin{align}
    \int_{E_1}\varphi(x)\gamma(dx, dy)
    &=\int_{\x}\varphi(x)\mu^{x_{1:T-1}}(dx_T)\cdots \mu^{x_{1}}(dx_2)(\mu_1-\nu_1)_{+}(dx_1)\\
    &=\int_{\x}\varphi(x)\overline{\mu}^{x_{1}}(dx_{2:T})(\mu_1-\nu_1)_{+}(dx_1)\\
    &\le \int_{\x}\varphi(x)\overline{\mu}^{x_{1}}(dx_{2:T})\abs{\mu_1-\nu_1}(dx_1)
    =\int_{\x}\varphi(x)\gamma^{(1)}(dx).
\end{align}
Similarly, for $2\le j\le T-1$,
\begin{align}
    &\int_{E_j}\varphi(x)\gamma(dx, dy)\\
    &=\int_{\x}\varphi(x)\overline{\mu}^{x_{1:j}}(dx_{j+1:T})(\mu^{x_{1:j-1}}-\nu^{x_{1:j-1}})_{+}(dx_j)
    \eta^{x_{1:j-2}}(dx_{j-1})\cdots \eta^{x_{1}}(dx_{2})\eta_1(dx_{1})\\
    &\le \int_{\x}\varphi(x)\overline{\mu}^{x_{1:j}}(dx_{j+1:T})\abs{\mu^{x_{1:j-1}}-\nu^{x_{1:j-1}}}(dx_j)
    \eta^{x_{1:j-2}}(dx_{j-1})\cdots \eta^{x_{1}}(dx_{2})\eta_1(dx_{1})\\
    &\le \int_{\x}\varphi(x)\overline{\mu}^{x_{1:j}}(dx_{j+1:T})\abs{\mu^{x_{1:j-1}}-\nu^{x_{1:j-1}}}(dx_j)
    \mu_{1:j-1}\wedge \nu_{1:j-1}(dx_{1:j-1})\\
    &=\int_{\x}\varphi(x)\gamma^{(j)}(dx).
\end{align}
The second inequality follows from the minimality of $\mu_{1:j-1}\wedge \nu_{1:j-1}$. Indeed,
\begin{align}
    \eta^{x_{1:j-2}}(dx_{j-1})\cdots \eta^{x_{1}}(dx_{2})\eta_1(dx_{1})
    \le \mu^{x_{1:j-1}}(dx_{j-1})\cdots \mu^{x_1}(dx_2)\mu_1(dx_1)
    =\mu_{1:j-1}(dx_{1:j-1})
\end{align}
and similarly,
\begin{align}
    \eta^{x_{1:j-2}}(dx_{j-1})\cdots \eta^{x_{1}}(dx_{2})\eta_1(dx_{1})
    \le \nu_{1:j-1}(dx_{1:j-1}).
\end{align}
Hence,
\begin{align}
    \eta^{x_{1:j-2}}(dx_{j-1})\cdots \eta^{x_{1}}(dx_{2})\eta_1(dx_{1})
    \le \mu_{1:j-1}\wedge\nu_{1:j-1}(dx_{1:j-1}).
\end{align}
The proof of the last case $j=T$ is essentially the same. Indeed,
\begin{align}
    \int_{E_j}\varphi(x)\gamma(dx, dy)
    &=\int_{\x}\varphi(x)(\mu^{x_{1:T-1}}-\nu^{x_{1:T-1}})_{+}(dx_T)
    \eta^{x_{1:T-2}}(dx_{T-1})\cdots \eta^{x_{1}}(dx_{2})\eta_1(dx_{1})\\
    &\le \int_{\x}\varphi(x)\abs{\mu^{x_{1:T-1}}-\nu^{x_{1:T-1}}}(dx_T)
    \mu_{1:T-1}\wedge \nu_{1:T-1}(dx_{1:T-1})\\
    &=\int_{\x}\varphi(x)\gamma^{(T)}(dx).
\end{align}
This proves the desired estimate.
\end{proof}

\begin{rmk}
It is noteworthy that the coupling $\gamma$ was used in \cite[Lemma 3.5]{eckstein2024computational} to establish the bound~\eqref{eq:atv and tv}. In fact, the coupling $\gamma$ attains the infimum in the definition of $\atv_{\varphi}(\mu, \nu)$ although we do not make use of this fact in the proof; see \cite[Lemma 3.2]{acciaio2025estimating}.
\end{rmk}

\begin{lem}\label{lem:gammaj H}
Recall $\gamma^{(j)}$ from Lemma~\ref{lem:tv+sum}. We have
\begin{align}
    \int_{\x}\varphi(x)\gamma^{(1)}(dx)
    \le \left(1+\log \int_{\x} e^{\varphi(x)^2}\mu(dx)\right)^{1/2}(2\ent{\nu_1}{\mu_1})^{1/2}
\end{align}
and for all $2\le j\le T$,
\begin{align}
    \int_{\x}\varphi(x)\gamma^{(j)}(dx)
    \le \left(1+\log \int_{\x} e^{\varphi(x)^2}\mu(dx)\right)^{1/2}
    \left(\int_{\x_{1:j-1}} 2\ent{\nu^{x_{1:j-1}}}{\mu^{x_{1:j-1}}}\nu_{1:j-1}(dx_{1:j-1})\right)^{1/2}.
\end{align}
\end{lem}
\begin{proof}
For $j\in \{1,2,\ldots, T-1\}$, let us define $\psi^{(j)}:\x_{1:j}\to [0,\infty)$ via
\begin{align}
    \psi^{(j)}(x_{1:j})
    :=\int_{\x_{j+1:T}}\varphi(x)\overline{\mu}^{x_{1:j}}(dx_{j+1:T}).
\end{align}
We also set $\psi^{(T)}= \varphi$. Since $a\mapsto e^{a^2}$ is convex, the Jensen inequality shows that
\begin{align}
    \int_{\x_j} e^{\psi^{(j)}(x_{1:j})^2} \mu^{x_{1:j-1}}(dx_j) 
    &\le \int_{\x_j}\left(\int_{\x_{j+1:T}}e^{\varphi(x)^2}\overline{\mu}^{x_{1:j}}(dx_{j+1:T})\right)\mu^{x_{1:j-1}}(dx_j)\\
    &=\int_{\x_{j:T}}e^{\varphi(x)^2}\overline{\mu}^{x_{1:j-1}}(dx_{j:T})\label{eq:psi bound}
\end{align}
for all $1\le j\le T$ with the convention that $\mu^{x_{1:0}}:=\mu_1$ and $\overline{\mu}^{x_{1:0}}=\mu$.

We begin by considering the case $j=1$. From Theorem~\ref{thm:bolley villani},
\begin{align}
    \int_{\x}\varphi(x)\gamma^{(1)}(dx)
    &=\int_{\x_{1}} \psi^{(1)}(x_1)\abs{\mu_1-\nu_1}(dx_1)\\
    &\le \left(1+\log \int_{\x_1} e^{\psi^{(1)}(x_1)^2} \mu_1(dx_1)\right)^{1/2}(2\ent{\nu_1}{\mu_1})^{1/2}.
\end{align}
The desired result follows from the bound \eqref{eq:psi bound} applied to $j=1$:
\begin{align}
    \int_{\x_1} e^{\psi^{(1)}(x_1)^2} \mu_1(dx_1)
    \le \int_{\x} e^{\varphi(x)^2} \mu(dx).
\end{align}
For $2\le j\le T$, observe that
\begin{align}
    &\int_{\x}\varphi(x)\gamma^{(j)}(dx)\\
    &=\int_{\x_{1:j-1}} \left(\int_{\x_{j}} \psi^{(j)}(x_{1:j}) \abs{\mu^{x_{1:j-1}}-\nu^{x_{1:j-1}}}(dx_{j})\right) \mu_{1:j-1}\wedge \nu_{1:j-1}(dx_{1:j-1}).\label{eq:gamma j psi bound}
\end{align}
Using Theorem~\ref{thm:bolley villani}, we bound
\begin{align}
    &\int_{\x_{j}} \psi^{(j)}(x_{1:j}) \abs{\mu^{x_{1:j-1}}-\nu^{x_{1:j-1}}}(dx_{j})\\
    &\le \left(1+\log \int_{\x_j} e^{\psi^{(j)}(x_{1:j})^2} \mu^{x_{1:j-1}}(dx_j)\right)^{1/2}(2\ent{\nu^{x_{1:j-1}}}{\mu^{x_{1:j-1}}})^{1/2}\\
    &\le \left(1+\log \int_{\x_{j:T}} e^{\varphi(x)^2} \overline{\mu}^{x_{1:j-1}}(dx_{j:T})\right)^{1/2}(2\ent{\nu^{x_{1:j-1}}}{\mu^{x_{1:j-1}}})^{1/2}
\end{align}
where the second inequality follows from \eqref{eq:psi bound}. Plugging this back into \eqref{eq:gamma j psi bound} and applying the Cauchy--Schwarz inequality
with
\begin{align}
    f(x_{1:j-1}):=1+\log \int_{\x_{j:T}} e^{\varphi(x)^2} \overline{\mu}^{x_{1:j-1}}(dx_{j:T}), \quad
    g(x_{1:j-1}):=2\ent{\nu^{x_{1:j-1}}}{\mu^{x_{1:j-1}}}
\end{align}
yields
\begin{align}
    &\left(\int_{\x}\varphi(x)\gamma^{(j)}(dx)\right)^2\\
    &\le\left(\int_{\x_{1:j-1}} f(x_{1:j-1})^{1/2}g(x_{1:j-1})^{1/2} \mu_{1:j-1}\wedge \nu_{1:j-1}(dx_{1:j-1})\right)^2\\
    &\le \left(\int_{\x_{1:j-1}} f(x_{1:j-1}) \mu_{1:j-1}\wedge \nu_{1:j-1}(dx_{1:j-1})\right)
    \left(\int_{\x_{1:j-1}} g(x_{1:j-1}) \mu_{1:j-1}\wedge \nu_{1:j-1}(dx_{1:j-1})\right)\\
    &\le \left(\int_{\x_{1:j-1}} f(x_{1:j-1}) \mu_{1:j-1}(dx_{1:j-1})\right)
    \left(\int_{\x_{1:j-1}} g(x_{1:j-1}) \nu_{1:j-1}(dx_{1:j-1})\right).
\end{align}
From the Jensen inequality applied to the concave function $a\mapsto \log(a)$,
\begin{align}
    \int_{\x_{1:j-1}} f(x_{1:j-1}) \mu_{1:j-1}(dx_{1:j-1})
    \le 1+\log \int_{\x} e^{\varphi(x)^2}\mu(dx).
\end{align}
This proves the desired estimate.
\end{proof}

\begin{proof}[Proof of Theorem~\ref{thm:adapted Bolley Villani}]
Combining Lemma~\ref{lem:tv+sum} and Lemma~\ref{lem:gammaj H}, we obtain
\begin{align}
    \atv_{\varphi}(\mu, \nu)
    &\le \tv_{\varphi}(\mu, \nu)
    +2\sum_{j=1}^{T} \int_{\x}\varphi(x)\gamma^{(j)}(dx)\\
    &\le \tv_{\varphi}(\mu, \nu)
    +2\left(1+\log \int_{\x} e^{\varphi(x)^2}\mu(dx)\right)^{1/2}\sum_{j=1}^T h_j^{1/2}
\end{align}
where $h_1:=2\ent{\nu_1}{\mu_1}$ and
\begin{align}
    h_j:=\int_{\x_{1:j-1}} 2\ent{\nu^{x_{1:j-1}}}{\mu^{x_{1:j-1}}}\nu_{1:j-1}(dx_{1:j-1}).
\end{align}
We apply Theorem~\ref{thm:bolley villani} to bound the first term $\tv_{\varphi}(\mu, \nu)$. To estimate the second term, we use the chain rule of entropy:
\begin{align}
    \sum_{j=1}^T h_j^{1/2}
    \le \sqrt{T}\left(\sum_{j=1}^T h_j\right)^{1/2}
    =\sqrt{T}\sqrt{2\ent{\nu}{\mu}}.
\end{align}
This concludes the proof.
\end{proof}

\begin{proof}[Proof of Corollary~\ref{cor:aw_1 H inequality} and Corollary~\ref{cor:general p}]
Note that
\begin{align}
    \a\w_p(\mu, \nu)^p
    \le 2^{p-1}\inf_{\pi \in \cplbc(\mu, \nu)} \int_{\x\times \x} (d(x, x_0)^p+d(y, x_0)^p)\ind{\{x\neq y\}}\pi(dx, dy).
\end{align}
By applying Theorem~\ref{thm:adapted Bolley Villani} to $\varphi(x):=\sqrt{\alpha}d(x, x_0)^p$, we immediately establish the estimates stated in Corollary~\ref{cor:aw_1 H inequality} and Corollary~\ref{cor:general p}. In particular, if $\int_{\x}e^{\alpha d(x, x_0)^2}\mu(dx)<\infty$ for some $x_0\in \x$ and $\alpha>0$, then $\mu$ satisfies the adapted $T_1$ inequality. This proves the implication from $(b)$ to $(c)$ in Corollary~\ref{cor:aw_1 H inequality}. The other direction is immediate from $\w_1\le \a\w_1$.

\end{proof}

\begin{small}
\bibliographystyle{abbrv}
\bibliography{transport_aw}
\end{small}

\end{document}